\documentclass[11pt]{amsart}
\usepackage{amsmath}
\usepackage{amsfonts}
\usepackage{amsthm}
\usepackage{amssymb}
\usepackage{mathrsfs}
\usepackage{latexsym}
\usepackage{verbatim}

\def\beb{\beta}
\def\rec{\mathrm{rec}}

\def\Tr{\mathrm{Tr}}

\def\im{\mathrm{image}}

\def\ad{\mathrm{ad}}

\def\NS{N}
\def\SC{C}

\def\sgn{\mathrm{sgn}}

\def\PGL{\mathrm{PGL}}

\def\ad{\mathrm{ad}}

\def\rbar{\overline{r}}

\def\Sym{\mathrm{Sym}}
\def\Sm{S}

\def\PSL{\mathrm{PSL}}
\def\SL{\mathrm{SL}}

\def\Hom{\mathrm{Hom}}
\def\F{\mathbf{F}}
\def\Q{\mathbf{Q}}
\def\GL{\mathrm{GL}}

\def\Z{\mathbf{Z}}
\def\C{\mathbf{C}}

\def\Remark{\noindent \bf  Remark\rm.\ }
\def\Ackn{\noindent \bf Acknowledgements\rm.\ }

\def\Gal{\mathrm{Gal}}
\def\GSp{\mathrm{GSp}}

\swapnumbers

\newtheorem{theorem}{Theorem}[section]
\newtheorem{question}[theorem]{Question}
\newtheorem{df}[theorem]{Definition}
\newtheorem{lemma}[theorem]{Lemma}

\newtheorem{corr}[theorem]{Corollary}
\newtheorem{conj}[theorem]{Conjecture}

\def\rhobar{\overline{\rho}}

\def\psibar{\overline{\psi}}

\def\eps{\epsilon}
\def\epsa{\varepsilon}

\def\OL{\mathcal{O}}

\def\Qbar{\overline{\Q}}

\def\Ind{\mathrm{Ind}}
\def\Frob{\mathrm{Frob}}

\def\R{\mathbf{R}}
\def\Sym{\mathrm{Sym}}

\begin{document}

\title{Even Galois Representations and the Fontaine--Mazur conjecture}
\author{Frank Calegari}
\thanks{Supported in part by NSF Career Grant DMS-0846285 and the Sloan Foundation.
}
\begin{abstract} We prove, under mild hypotheses, that
there are no irreducible  two-dimensional ordinary \emph{even} Galois representations
of $\Gal(\Qbar/\Q)$ with distinct Hodge--Tate weights.
This is in accordance with the Fontaine--Mazur conjecture. 
If $K/\Q$ is an imaginary quadratic field, we also prove (again, under certain hypotheses)
that
$\Gal(\Qbar/K)$ does not admit irreducible two-dimensional ordinary
Galois representations of non-parallel weight. 
\end{abstract}
\maketitle

\section{Introduction}

Potential modularity has proved to be a powerful tool for studying arithmetic questions
in the Langlands program~\cite{Taylor, Khare}.
In this note,  we show how this circle of ideas can be employed in a novel
way to deduce
some new instances of the Fontaine--Mazur conjectures.

 The conjecture of Fontaine and Mazur  (\cite{FM}, p.41 first edition, p.190 second edition) is the following:
 \begin{conj}[Fontaine--Mazur] Let  \label{conj:FM}
  $$\rho: G_{\Q} \rightarrow \GL_2(\Qbar_p)$$
 be an irreducible representation which is unramified except at a finite number of primes and which is not
 the Tate twist of an even representation which factors through a  finite quotient group of $G_{\Q}$.
 Then $\rho$ is associated to a cuspidal newform $f$ if and only if $\rho$ is potentially semi-stable at $p$.
 \end{conj}
 
Let  $D_p \subset G_{\Q}:=\Gal(\Qbar/\Q)$ denote a decomposition group at $p$. 
Assume that the Hodge--Tate weights of $\rho$ are distinct, and so,
in particular,  the Hodge--Tate weights
of any twist of $\rho$ are also distinct. 
Any two-dimensional representation of $D_p$ with finite image is Hodge--Tate with Hodge--Tate weights $(0,0)$.
Hence,  no twist of $\rho$ can have finite
image, and
 Conjecture~\ref{conj:FM} predicts that $\rho$ is modular.
It is well known that Galois representations arising from classical modular forms are 
 \emph{odd}, namely, $\det(\rho(c)) = -1$ for a complex conjugation $c \in G_{\Q}$.
 We deduce that if $\rho$ is \emph{even}
(that is, $\det(\rho(c)) = 1$), and the Hodge--Tate weights of $\rho$ are distinct, then the conjecture of Fontaine and Mazur predicts that
$\rho$ does not exist.  

Up to conjugation, the image of $\rho$ lands in $\GL_2(\OL)$ where
$\OL$ is the ring of integers of some finite extension $L/\Q_p$ (see Lemme~2.2.1.1 of~\cite{BM}). Let $\F$ denote the residue field
of $\OL$.
 We prove the following.
\begin{theorem} \label{theorem:even} Let $E$ be a totally real field, and let
$\rho: G_{E} \rightarrow \GL_2(\OL)$ be a continuous
irreducible Galois representation unramified outside finitely many primes. Suppose
that $p > 7$, and, furthermore,
that
\begin{enumerate}
\item $\rho |D_v$ is ordinary with distinct Hodge--Tate weights for all $v|p$.
\item The residual representation $\rhobar$ has
image containing $\SL_2(\F_p)$.
\item $[E(\zeta_p)^{+}:E] > 2$.
\end{enumerate}
Then $\det(\rho(c)) = -1$ for any complex conjugation $c$.
\end{theorem}

In light of the recent proof of Serre's conjecture~\cite{Khare,Kisin} and
modularity lifting theorems
for ordinary representations~\cite{G}, we immediately deduce the following corollary.

\begin{corr} Let
$\rho: G_{\Q} \rightarrow \GL_2(\OL)$ be a continuous
irreducible Galois representation unramified outside finitely many primes. Suppose
that $p > 7$, and, furthermore,
that
\begin{enumerate}
\item $\rho |D_p$ is 
ordinary, with distinct Hodge--Tate weights.
\item The residual representation $\rhobar$ has
image containing $\SL_2(\F_p)$.
\end{enumerate}
Then $\rho$ is modular.
\end{corr}

\Remark It should be remarked that \emph{some} $p$-adic Hodge
theory condition is necessary for the proof Theorem~\ref{theorem:even}. For example, Corollary~1(b) of Ramakrishna~\cite{R}
shows that there exist infinitely many even
surjective representations
$\rho: G_{\Q} \rightarrow \SL_2(\Z_7)$ unramified outside
a finite set of primes.

\medskip

The idea behind this theorem is simple: It suffices to show that $\rho$
is \emph{potentially} modular over some totally real field. Since modular
representations are odd, the theorem follows immediately. (The actual
argument is somewhat more circuitous.)

\medskip

There are other circumstances in which one expects (following Fontaine-Mazur)
the nonexistence of semistable Galois representations and for which the
methods of this paper also apply.
Let $K/\Q$ be an imaginary quadratic field, and let
$$\rho: G_K \rightarrow \GL_2(\OL)$$
be a continuous irreducible geometric Galois representation.
The the most general modularity conjectures
(following Fontaine--Mazur, Langlands, Clozel, and others)
predict
the existence of a cuspidal automorphic representation
$\pi$ for $G =  \GL(2)/K$ such that for all finite places $v \nmid p$ of $K$,  $\pi_{v}$ is determined
by $\rho|_{K_v}$ via the local Langlands correspondence.
Suppose that $p$ splits in $K$, and that the local representations
$\rho|_{D_v}$ for $v|p$ have Hodge--Tate weights $(0,m)$ and
$(0,n)$ respectively, for positive integers $m$, $n$.  The Hodge--Tate weights
conjecturally determine the corresponding infinity type $\pi_{\infty}$ of $\pi$. If $m \ne n$, however, a vanishing
theorem 
of Borel and Wallach implies that no such cuspidal $\pi$ can exist.
We prove the following result in this direction.

\begin{theorem} \label{theorem:parallel2} Let $\rho: G_K \rightarrow \GL_2(\OL)$
be a continuous irreducible geometric Galois representation.
Suppose that $p > 7$ splits in $K$, and, furthermore, that
\begin{enumerate}
\item $\rho |D_v$ is 
ordinary for $v|p$,  with    Hodge--Tate weights
$(0,m)$ and $(0,n)$, where $m,n > 0$. 
\item The residual representation $\rhobar$ 
has image containing $\SL_2(\F_p)$, 
and the projective representation
$\mathrm{Proj}(\rhobar):
G_K \rightarrow \mathrm{PGL}_2(\F)$  does not extend to $G_{\Q}$.
\end{enumerate}
Then $m = n.$
\end{theorem}

\Ackn I would like to thank Thomas Barnet-Lamb,  Jo\"{e}l Bella\"{\i}che, Matthew Emerton, Toby Gee,  Michael Harris, Dinakar Ramakrishan, Chris Skinner, and Richard Taylor for
helpful conversations. I would also like to thank Thomas Barnet--Lamb, David Geraghty,
and Richard Taylor  for keeping me informed
of their upcoming work~\cite{TBL,TBL2,BLGHT,G}.
 I would like to thank Dinakar  Ramakrishnan for explaining  to me the contents of his forthcoming paper~\cite{Dinakar}, and  A.~Raghuram
 for directing me  towards his recent paper with Asgari~\cite{AR}.
 I would especially like to thank Richard
 Taylor for acquiescing  to my repeated requests to write the article~\cite{TP}, allowing for a ``morally
 correct'' proof of the main theorem.
 Finally, I would like to thank the patience of the referee, whose insistence on my providing arguments in lieu of glib 
 assertions led to the correction of a few minor inaccuracies as well as greater readability of the paper.

\medskip

Theorem~\ref{theorem:even} is proven in section~\ref{section:even}, and
Theorem~\ref{theorem:parallel2} is proven in section~\ref{section:ramakrishnan}.
Recall the abbreviations RAESDC  and RACSDC for an automorphic form $\pi$ for
$\GL(n)$ stand for regular, algebraic, essentially-self-dual, and cuspidal and
regular, algebraic, conjugate-self-dual, and cuspidal, respectively.
 
\section{The Tensor Representation}
\label{section:ramakrishnan}

In this section, we prove
Theorem~\ref{theorem:parallel2}.
The following proof is inspired by an idea
of Ramakrishnan~\cite{Dinakar} to construct the Galois representations associated to
$\Sym^2 \pi_K$, where $\pi_K$ is a regular algebraic automorphic form
for $\GL(2)/K$ for an imaginary quadratic field $K$,
 without assuming any conditions on the central character.

\medskip

Let $\rho$ satisfy the first condition of Theorem~\ref{theorem:parallel2}.
Our proof is by contradiction. Assume that the Hodge--Tate weights
of $\rho$ at the primes $v|p$ are $(0,m)$ and 
$(0,n)$ respectively, where $m,n > 0$ and $m \ne n$.
Let us consider the representation $\rho \otimes \rho^c: G_{K} \rightarrow \GL_4(\OL)$;
it lifts to a representation of $G_{\Q}$ that is unique up to twisting by the quadratic
character $\eta$ of $\Gal(K/\Q)$.

\begin{lemma} The extension $\psi$ of $\rho \otimes \rho^c$ to $G_{\Q}$ satisfies
$\psi \simeq \psi^{\vee} \chi$, where $\chi |_{K} = \det(\rho) \det(\rho^c)$ and \label{lemma:bigrefer}
$\chi(c) = \epsa = + 1$ for any complex conjugation $c \in G_{\Q}$.
\end{lemma}

\begin{proof} Since $\psi|_{K}$ and $\psi^{\vee} \det(\rho) \det(\rho^c)|_K$ are isomorphic
and extend to $\Q$ uniquely up to twisting by $\eta$, the lemma is obvious up
to the sign of $\epsa$. If  
$\epsa = -1$, then  $\psi^{\vee}(c) = - \psi(c)$ and hence $\mathrm{Tr}(\psi(c)) = 0$.
From the definition of $\psi$, however,
$$\psi(c) \sim  \pm \left( \begin{matrix} 1 & 0 & & \\ 0 & 1 & & \\ & & 0 & 1 \\ & & 1 & 0 \end{matrix} \right),$$
and hence $\epsa = +1$.
\end{proof}

\begin{theorem} The representation $\psi$ is modular for $\GL(4)/F^{+}$ for
some totally real field $F^{+}$. \label{theorem:construct}
\end{theorem}

\begin{proof} 
The representation $\psi$ is ordinary with  Hodge--Tate weights
$(0,m,n,m+n)$, which are distinct by assumption.
The isomorphism
$\psi \simeq \psi^{\vee} \chi$ gives rise to a pairing $\langle x,y \rangle$ on the
vector space $L^4$ associated to $\psi$ such that
$\langle \sigma x, \sigma y \rangle = \chi(\sigma) \langle x,y \rangle$.
 Because $\wedge^2 \psi$ is irreducible,
this pairing is symmetric. Thus $\langle y,x \rangle = \epsa \langle x,y \rangle$,
where $\epsa =  +1 = \chi(c)$. Hence,
we may apply Theorem~7.5 of~\cite{BLGHT} to deduce the existence of
a RAESDC representation $\Pi_{F^{+}}$ for $\GL(4)/F^{+}$ associated to $\psi$. In order to apply this theorem, we also need
to assume that $p > 8$, and that, extending scalars if necessary, $\psibar$ has $2$-big image, a calculation that
we relegate to section~\ref{section:appendix}. 
\end{proof}

\subsection{A digression about normalizations} \label{section:ll}
The modularity lifting theorems of~\cite{BLGHT} and~\cite{CHT} associate to certain
Galois representations
 $\rho:G_{F^{+}} \rightarrow \GL_n(\Qbar_p)$ an algebraic automorphic form
$\Pi_{F^{+}}$ for $\GL(n)/F^{+}$. Here \emph{algebraic} means $C$-algebraic (cf.~\cite{buzz} and \S4.2
of~\cite{Clozel}; here $C$ refers either
to ``cohomological'' 
or ``Clozel''). On the other hand, it is sometimes useful to consider the normalized
twist $\Pi_{F^{+}} | \cdot |^{(n-1)/2}$. Each normalization has its advantage;
  the former arises
 naturally when one constructs automorphic forms via cohomology, whereas
 the latter has the property that the restriction
of the corresponding Langlands parameter at a place $v|\infty$ to $W_{\C} = \C^{\times}$ is an
algebraic character --- that is, the  normalization $\Pi_{F^{+}} | \cdot |^{(n-1)/2}$
is $L$-algebraic (where $L$ refers either to
``$L$-group'' or ``Langlands''). If $v$ is a finite place of $F^{+}$ not dividing $p$, then the
automorphic form $\Pi_{F^{+}}$ associated to
 $\rho$ by the main theorems of~\cite{BLGHT} and~\cite{CHT} (see
 \S3.1 of~\cite{CHT}) are related as follows:
$$ \iota \circ (\rho |_{W_{F^{+}_v}})^{\mathrm{ss}} = \rec \left(\Pi^{\vee}_{F^{+},v} | \cdot |^{\frac{1-n}{2}} \right)^{\mathrm{ss}},$$
where $\rec$ is the local Langlands correspondence (see~\cite{HT}), $\mathrm{ss}$ denotes semisimplification,
and $\iota$ is an isomorphism $\Qbar_p \simeq \C$. In particular, the reciprocity map associates the
normalization $\Pi_{F^{+}} | \cdot |^{(n-1)/2}$ to the dual representation $\rho^{\vee}$. In the sequel, when
we refer to the \emph{Langlands parameters} at infinity  associated to $\Pi_{F^{+}}$, we shall literally mean
the archimedean Langlands parameters associated  to the twist $\Pi_{F^{+}} | \cdot |^{(n-1)/2}$. This allows us
to work with parameters that more  naturally  reflect the properties of the Galois representation $\rho$,  moreover,
these parameters are somewhat more conveniently  behaved under functorality.
If $\eps$ is the $p$-adic cyclotomic character of $G_{\Q}$, we follow the convention that  the Hodge--Tate weight of
$\eps|_{G_{\Q_p}}$  to be $1$.
The corresponding automorphic representation of $\GL(1)/\Q$ is $| \cdot |^{-1}$, and we have
$$\iota \circ (\eps|_{W_{\Q_p}})^{\mathrm{ss}} = \rec(|\cdot|^{-1,\vee}) = \rec(|\cdot|).$$

\subsection{The proof of Theorem~\ref{theorem:parallel2}}
Let $\Pi_{F^{+}}$ denote the automorphic form for $\GL(4)/F^{+}$   associated to $\psi$ 
whose existence was established by
Theorem~\ref{theorem:construct}.
For any infinite place $v | \infty$ of $F^{+}$,
 the associated archimedean
Langlands parameter is a $4$-dimensional
representation $\sigma_{v}$   of the real Weil group $W_{\R}$ 
for each infinite place $v$ of $F^{+}$. 
Since $\Pi_{F^{+}}$ is  algebraic, the restriction of $\sigma_{v}$ to
$\C^{\times}$ is of the form $z \mapsto z^p \overline{z}^{q}$ for integers 
$p$ and $q$. By purity (Lemma~4.9 of~\cite{Clozel}), the sum $p + q$ only depends on $\pi$ and $v$.
(Since the Galois representation descends to $\Q$,
$\sigma_v$ only depends on $\pi$, as we shall see below.)
Since $\pi$ is regular and $\psi$ has Hodge--Tate weights $(0,m,n,m+n)$ at
every $v|p$,  it follows (with our normalizations)
that
$\sigma_v$ is of the form
$$\sigma_v \simeq I(z^{m+n}) \oplus I(z^n \overline{z}^{m}),$$
where  $I(\xi)$ denotes the induced representation
of $\xi$ from $\C^{\times}$ to $W_{\R}$. 
Here we have invoked the compatibility at $v|p$ between the infinity
type of $\Pi_{F^{+}}$ and the corresponding Hodge--Tate weights (see
 Theorem~1.1(4) of~\cite{BLGHT}).
Since $I(\xi) \otimes
I(\zeta) = I(\xi \zeta) \oplus I(\xi \overline{\zeta})$, it follows
(cf. (5.8) of~\cite{irr3}) that
$$\wedge^2(\sigma_v) \simeq I(z^{n} \overline{z}^{n+2m}) \oplus
I(z^{m} \overline{z}^{2n+m}) \oplus \sgn^{m+n+1} |\cdot|^{m+n} \oplus
\sgn^{m+n+1} | \cdot|^{m+n}.$$
A result of Kim~\cite{K} implies the existence of the
exterior square $\wedge^2 \Pi_{F^{+}}$
for $\GL(6)/F^{+}$, whose construction is compatible with functoriality at all places except
possibly those
dividing $2$ and $3$. In particular, the infinity type of $\wedge^2 \Pi_{F^+}$ at
$v$ is $\wedge^2(\sigma_v)$. Moreover, the work of Kim also implies
that
 the representation
 $\wedge^2 \Pi_{F^{+}}$ can be taken to be isobaric.

\begin{lemma} The base change of $\Pi_{F^{+}}$ to any solvable Galois extension of
$F^{+}$
is cuspidal. \label{lemma:remainscuspidal}
\end{lemma}

\begin{proof} The existence of the base change to any solvable
extension follows by 
 repeated applications of Theorem~4.2 (p.202)  of~\cite{AC}.
By assumption, the representation
$\psibar$ and hence $\psi$ remains irreducible after restricting to any solvable
extension, and thus the restriction of $\psi$ to any solvable extension
admits no isomorphisms to   any non-trivial self-twist.
Thus, we  are always in the setting of Theorem 4.2(a) of \emph{ibid}, and we
deduce the corresponding base changes are cuspidal.
\end{proof}

\begin{lemma} The automorphic form  $\wedge^2 \Pi_{F^{+}}$ is cuspidal
for $\GL(6)$.
\end{lemma}

\begin{proof} 
If $\wedge^2 \Pi_{F^{+}}$ is not cuspidal, then, since it is isobaric, we may write
$$\wedge^2 \Pi_{F^{+}} \simeq \boxplus^{m}_{j=1} \beb_j, \quad m > 1,$$
where each $\beb_j$ is a cuspidal automorphic representation of
$\GL(n_j)/F^{+}$ with $\sum_{j} n_j = 6$ and $n_i \le n_j$ if $i \le j$. By
comparison with the infinity type $\wedge^2 \sigma_v$, the
forms $\beb_j$  will necessarily be algebraic.
Theorem~1.1 of~\cite{AR} gives gives a list of necessary (and sufficient)
conditions for $\wedge^2 \Pi_{F^{+}}$ to be non-cuspidal.
We show why none of these possibilities may occur.
\begin{enumerate}
\item $\Pi_{F^{+}} = \pi_1 \boxtimes \pi_2$, where $\pi_1$ and $\pi_2$
are cuspidal for $\GL(2)/F^{+}$. If either $\pi_1$ or $\pi_2$ does not remain cuspidal
over some quadratic extension of $F^{+}$, then  $\Pi_{F^+}$ does not remain cuspidal either,
contradicting Lemma~\ref{lemma:remainscuspidal}.  Hence, we may assume
that $\Sym^2 \pi_1$ and $\Sym^2 \pi_2$ are both cuspidal. 
It follows that
$$\wedge^2 \Pi_{F^{+}} = \beb_1 \boxtimes \beb_2 = \chi_2
\otimes  \Sym^2 \pi_1 \boxtimes \chi_1 \otimes  \Sym^2 \pi_2,$$
where $\chi_1$ and $\chi_2$ are characters. In particular,  the infinity type
of $\beb_1$ and $\beb_2$ must be
$$ I(z^{n} \overline{z}^{n+2m}) \oplus \sgn^{m+n+1} | \cdot|^{m+n} \quad
\text{and} \quad
I(z^{m} \overline{z}^{2n+m}) \oplus \sgn^{m+n+1} |\cdot|^{m+n}$$
respectively. Both these characters are algebraic and regular. Moreover,
 $\beb_1$ and $\beb_2$ are essentially self-dual (via their
identifications up to twist with symmetric squares), and $F^{+}$ is totally
real. It follows from the main result of~\cite{zeta}
that $\beb_1$ and $\beb_2$ may be associated to three dimensional
(essentially self-dual)  Galois representations of $G_{F^{+}}$. Yet
$\wedge^2 \Pi_{F^{+}}$ is associated to the Galois representation 
$\wedge^2 \psi|_{G_{F^{+}}}$, which is easily seen to be irreducible, and these
facts are incompatible.
\item $\Pi_{F^{+}}$ is the Asai transfer of a dihedral  cuspidal
automorphic representation $\pi_E$ of $\GL(2)/E$ for some quadratic extension
$E/F^{+}$. We deduce that the base change of $\pi_E$ to
some  CM extension $H/E$ of degree two is no longer cuspidal, and thus
that the base change of $\Pi_{F^{+}}$ to the (solvable) Galois
closure of $H$ over $F^{+}$
is also not cuspidal. This contradicts Lemma~\ref{lemma:remainscuspidal}.
\item $\Pi_{F^{+}}$ is of symplectic type. By the main theorem of~\cite{BC} (cf.~\cite{CC}), we may deduce
that the symplectic/orthogonal alternative for the automorphic form
determines and is determined by the Galois representation. Yet
the Galois  representation is of orthogonal type.
\item $\Pi_{F^{+}}$ is the automorphic induction from some quadratic 
extension $E/F^{+}$.
If $\Pi_{F^{+}}$ arises via automorphic induction, then it is isomorphic
to a self twist by a quadratic character. It follows that the base change
of $\Pi_{F^{+}}$ to $E$ is not cuspidal,
contradicting Lemma~\ref{lemma:remainscuspidal}.
\end{enumerate}
\end{proof}

On the level of Galois representations, if we let $F = F^{+}.K$ then
$$\wedge^2 \psi |_{F} = \Sym^2 (\rho)  \det(\rho^c) \oplus \Sym^2(\rho^c)  \det(\rho).$$
By multiplicity one for $\GL(6)$~\cite{J}, it follows that
$\wedge^2 \Pi_{F^{+}} \simeq \wedge^2 \Pi_{F^{+}} \otimes \eta$, where $\eta$ is the quadratic character
of $F/{F^{+}}$. In particular,  from Theorem~4.2 of~\cite{AC},
 $\wedge^2 \Pi_{F^{+}}$ is the automorphic
induction of an automorphic form for $\GL(3)/F$, which we shall denote by 
$\Sm(\pi_F)$. From the description of the infinity type $\wedge^2 \sigma_v$
of $\wedge^2 \Pi_{F^{+}}$ given above, we deduce that the infinity type
of $\Sm(\pi_F)$ at $v$ is one of four possibilities:
$$\begin{aligned}
z \mapsto & \ z^n \overline{z}^{n+2m} \oplus z^m \overline{z}^{2n+m}
\oplus z^{m+n} \overline{z}^{m+n}, \\
z \mapsto & \ z^n \overline{z}^{n+2m} \oplus z^{2n+m} \overline{z}^{m}
\oplus z^{m+n} \overline{z}^{m+n},  \\
z \mapsto & \ z^{n+2m} \overline{z}^{n} \oplus z^m \overline{z}^{2n+m}
\oplus z^{m+n} \overline{z}^{m+n},  \\
z \mapsto & \ z^{n+2m} \overline{z}^{n} \oplus z^{2n+m} \overline{z}^{m}
\oplus z^{m+n} \overline{z}^{m+n}. \end{aligned}$$
These correspond to the only four partitions of (the six characters occuring in)
$\wedge^2 \sigma_v |_{\C^{\times}}$ into  two sets of three characters which are permuted
by the involution $z \mapsto \overline{z}$. This calculation uses the fact that
$\{0,m,n,m+n\}$ are all distinct. 
The notation $\Sm(\pi_F)$  is
 meant to suggest the existence of an automorphic form $\pi_F$ for $\GL(2)/F$
associated to $\rho|_{G_F}$; if such a $\pi_F$ existed then $\Sym^2 \pi_F$ would be isomorphic
(up to twist) 
 to $\Sm(\pi_F)$. It is not  necessary for our arguments, however,
to establish the existence of such a $\pi_F$.
We see explicitly that the infinity type of $\Sm(\pi_F)$ is regular algebraic
and cohomological. 
Moreover, we deduce
  that  the infinity type
  is not preserved by the Cartan involution (which effectively replaces
  $z$ by $\overline{z}$ in this case --- once more noting that $\{0,m,n,m+n\}$ are
  all distinct by hypothesis), and thus
 $\Sm(\pi_F)$ cannot exist, by
  Borel--Wallach 
  (Theorem 6.7, VII, p.226~\cite{BW}). 
  As stated, this theorem applies only in the compact case.
  However,  the same proof (via vanishing
  of $(\mathfrak{g},K)$-cohomology)  applies more generally
  to the \emph{cuspidal} cohomology in the non-compact case, which
  also may be computed via
  $(\mathfrak{g},K)$-cohomology.
    \qed

 \section{Even Representations}
\label{section:even}

Let $E/\Q$ be a totally real field.
Recall that a representation $r: G_E \rightarrow \GL_n(\OL)$ is \emph{odd} if, for any complex conjugation $c \in G_E$, 
$\Tr(\rho(c)) = -1$, $0$, or $1$. We start by noting the following:

\begin{theorem}  Let $n$ be odd. Let $E$ be a totally real field, and let \label{theorem:evenodd}
$r: G_{E} \rightarrow  \GL_n(\OL)$ be a continuous irreducible Galois representation
unramified outside finitely many primes.
Let $\chi$ be a finite order character of $G_{E}$ that is unramified at all $v|p$.
 Suppose that $p > 2n$, and, furthermore, that
\begin{enumerate}
\item $r$ is self-dual up to twist: $r \simeq r^{\vee} \eps^{1 - n} \chi$.
\item $r$ is ordinary for all $v|p$ with distinct Hodge--Tate weights.
\item $\rbar$ has $2$-big image, in the sense of~\cite{BLGHT}, Definition~7.2.
\item The fixed field of $\ad(\rbar)$ does not contain
$E(\zeta_p)^{+}$.
\item $(\det \rbar)^2 = \eps^{n(1-n)}   \mod p$.
\end{enumerate}
Then there exits a totally real field $F^{+}/\Q$ such that
the restriction $r |_{F^{+}}$ is modular, i.e., associated to a RAESDC
form $\Pi_{F^{+}}$.
\end{theorem}

\begin{proof} Taking determinants of both sides of the relation
$r \simeq r^{\vee} \eps^{1-n} \chi$, we deduce that $\eps^{1-n} \chi(c) = + 1$
for any complex conjugation $c$.
Since $r$ is irreducible and self-dual up to twist, we may deduce that
$r$ preserves a non-degenerate pairing $\langle x,y \rangle$. Since
$n$ is odd, moreover, this pairing must be symmetric.
Hence, we
may apply    Theorem~7.5 of~\cite{BLGHT} (with $\eps = +1$) to deduce the result.
\end{proof}

\Remark It should be possible to prove this theorem for $p > n$ as follows.
Suppose that $r$    corresponds to a RACSDC automorphic
$\Pi_F$ for $\GL(n)/F$ for some CM field $F$. By
Lemma~4.3.3 of~\cite{CHT}, we may deduce that $\Pi_F$ descends to a RAESDC
form $\Pi_{F^{+}}$ for $\GL(n)/F^{+}$.
Hence,
in light of modularity lifting theorems of Geraghty~\cite{G} (in particular, Theorem~5.3.2), it suffices to prove that
$\rbar|_{G_F}$ is modular for some CM field $F$
which is  sufficiently disjoint from $E(\zeta_p)^{+}$.
One approach to proving modularity theorems of this type arises in
the work of Barnet-Lamb (\cite{TBL2}, Proposition 7).
As written, this theorem requires some extra assumptions, in particular, that $\rbar$ is crystalline
and that the residue field of $\OL$ is $\F_p$ (the fact that $n$ is odd guarantees that all sign
conditions are satisfied).
However, combining this approach with recent advances
(particularly, generalizing the results about the monodromy of the Dwork family proved 
in~(\cite{BLGHT}~\S4, \S5) to the more general setting of~\cite{TBL2}), these conditions
can presumably be removed in the ordinary case.   The reason that one obtains
a better bound on $p$ is that this method  requires modularity lifting theorems for
$\GL(n)$ (over a CM field) rather than modularity theorems for
$\GL(2n)$  (over a totally real field) as in the proof
of Theorem~7.5 of~\cite{BLGHT}.

\medskip 

\noindent \bf Proof of Theorem~\ref{theorem:even}.\rm \ 
Suppose that $\rho$ satisfies the
conditions of Theorem~\ref{theorem:even}. I claim that
$r = \Sym^2(\rho) \otimes (\eps^{-1} \det(\rho)^{-1}) = \ad^0(\rho) \otimes \eps^{-1}$ satisfies the conditions
of Theorem~\ref{theorem:evenodd}. In particular:
\begin{enumerate}
\item $r$ is self-dual up to twist, because $\ad^0(\rho)$ is self-dual.
\item $r$ is ordinary, because $\rho$ is ordinary (by assumption).
\item $\rbar$ has $2$-big image. By corollary~2.5.4 of~\cite{CHT}, $\rbar$ has big-image.
In fact,  the same proof applies to show that $\rbar$ has $2$-big image. The difference
in the definition of $2$-big image and big image is the requirement that the element $h \in H$
with eigenvalue $\alpha$ 
arising in the definition of bigness (Definition~2.5.1 of~\cite{CHT}, see also
section~\ref{section:appendix} of this paper)  
has the property that if $\beta$ is any other generalized eigenvalue of $h$, then $\alpha^2 \ne \beta^2$.
In the proof of Lemma~2.5.2 of~\cite{CHT}, the element $h$ is taken to be a generator $t$ of the $\F_p$-split
torus of $\SL_2(\F_p)$. The image of $t$  acting via the adjoint representation of $\SL_2(\F_p)$  is
$$\mathrm{diag}(\delta^2,1,\delta^{-2}),$$
where $\delta$ is a generator of $\F^{\times}_p$. If $\alpha$ and $\beta$ are two distinct eigenvalues
of this matrix, then $\alpha^2 = \beta^2$ implies that $\delta^8 = 1$. Yet $\delta$ generates $\F^{\times}_p$ and
thus has order $p - 1$, which does not divide $8$ if $p \ge 7$.

It should be noted that Lemma~2.5.2 of~\cite{CHT} is not exactly
correct as stated --- the requirement on $l$ in \it{ibid}.\ \rm  should be $l > 2n+1$ rather
than $l > 2n-1$. The issue is the appeal to~\cite{CPS}; the group
$H^1(U,\mathrm{Symm}^{2i})^B$ vanishes for $l > 2n + 1$ rather than $l > n+1$.
Of relevance to this paper is that $H^1(\SL_2(\F_7),\Sym^4 \F^2) \ne 0$; this is why
we assume that $p > 7$. In fact, the main theorem of our paper still
holds when $p = 7$ under the stronger assumption
that  $\SL_2(\F_q) \subseteq \im(\rhobar)$ for some $q > p$.
\item The fixed field of $\ad(\rbar)$ does not contain $E(\zeta_p)^{+}$.
We are assuming the image of $\rhobar$ contains $\SL_2(\F_p)$, and in particular that
$$\SL_2(\F_p) \subset \im(\rhobar) \subset \GL_2(\F).$$
It follows that the image of 
$\ad^0(\rhobar)$ contains $\PSL_2(\F_p)$ and is contained in $\PGL_2(\F)$. Consequently
(by the classification of subgroups of $\PGL_2(\F)$), the
image is isomorphic to either $\PSL_2(k)$ or $\PGL_2(k)$ for some
 $\F_p \subset k \subset \F$. Since $\#k \ge 7$,  $\PSL_2(k)$  is a simple group.
Since $\rbar$ is a twist of $\ad^0(\rhobar)$, the image of $\ad(\rbar)$ is also isomorphic to
$\PSL_2(k)$ or $\PGL_2(k)$.  In particular, the maximal solvable extension $E'$ of $E$ contained
inside $E(\ad(\rbar))$ has degree at most $2$ over $E$.
Hence
$$E(\ad(\rbar)) \cap E(\zeta_p)^{+} = E' \cap E(\zeta_p)^{+} \subset E'.$$
Yet $[E':E] \le 2$, and by assumption, $[E(\zeta_p)^{+}:E] > 2$. Thus $E(\ad(\rbar))$ does not
contain $E(\zeta_p)^{+}$.
\item Since $\det(\ad^0(\rhobar))$ is trivial, $\det(\rbar)^2 = \det(\eps^{-1})^2 = \eps^{-6} = \eps^{3(1-3)}$.
\end{enumerate}
We deduce that $r$ is modular over some totally real field $F^{+}$.
Using the same normalizations discussed in section~\ref{section:ll}, we conclude that
there exists an automorphic form $\Pi_{F^{+}}$ for $\GL(3)/F^{+}$ such that,
for all finite $v$ in $F^{+}$ not dividing $p$, 
$$ \iota \circ (r |_{W_{F^{+}_v}})^{\mathrm{ss}} = \rec \left(\Pi^{\vee}_{F^{+},v} | \cdot |^{-1} \right)^{\mathrm{ss}}.$$
Since $\rec(|\cdot|^{-1}) = \eps^{-1}$, we may twist both sides to deduce that
$$ \iota \circ (\ad^0(\rho) |_{W_{F^{+}_v}})^{\mathrm{ss}} = \rec \left(\Pi^{\vee}_{F^{+},v} \right)^{\mathrm{ss}}.$$
By multiplicity one for $\GL(3)$ (\cite{J}), 
we immediately deduce (by considering the Galois representation
$\ad^0(\rho)$) that $\Pi^{\vee}_{F^+} \simeq \Pi_{F^+}$ and that $\Pi_{F^+}$ has trivial central character.
If follows from Theorem~A and Corollary~B of~\cite{Appendix} that $\Pi_{F^+}$ is
 the symmetric square of a RAESDC automorphic form $\pi_{F^+}$
for $\GL(2)/F^+$, that is, a Hilbert modular form.
Such automorphic forms are known to admit a $p$-adic Galois representation,
which must be equal to $\rho|_{F^+}$ up to twist. Yet the Galois representations associated to Hilbert modular
forms are odd, and thus $\rho$ is odd.
\qed

\medskip

Instead of appealing to functorial properties of the symmetric square lift,
one may appeal to the following theorem:

\begin{theorem}[Taylor,~\cite{TP}] Let $F^{+}$ be a totally real field, let 
\label{theorem:totallyodd2}
$\Pi_{F^+}$ be a RAESDC automorphic form for $\GL(n)/F^{+}$ for odd $n$, and
let $r$ be a $p$-adic Galois representation associated to $\Pi_{F^+}$.
Assume that
$r$ is irreducible. \label{theorem:TayTay}
Then $r$ is odd.
\end{theorem}

\Remark By Theorem~\ref{theorem:totallyodd2},  one may deduce
immediately from the modularity of $r$ that $\Sym^2(\rho)$ and hence $\rho$ are odd.

\medskip

\Remark If  the analog of Theorem~\ref{theorem:TayTay} was known for \emph{even} $n$,
this would also lead to a different proof of Theorem~\ref{theorem:parallel2}. Namely,
the Galois representation  $\psi = \rho \otimes \rho^c$ is shown to be associated
to a RAEDSC form  $\Pi_{F^{+}}$ for $\GL(4)/F^{+}$,
and yet $\mathrm{Tr}(\psi(c)) = \pm 2$.

\medskip

\Remark Although the Fontaine--Mazur conjecture predicts
that Theorem~\ref{theorem:evenodd} should continue hold when $n$ is even, the
argument above cannot (directly) be made to work. For $n$ even, the representations
$\rbar$ will \emph{not}, in general, be potentially modular (in the sense we are using)
over any CM field $F$, because the Bella\"{\i}che--Chenevier sign of the
Galois representation (see~\cite{BC}) may be $-1$.
Indeed, when $n = 2$, all representations are self-dual up to twist, and so an even
representation $\rhobar: G_E \rightarrow \GL_2(\F)$ is \emph{never} potentially
modular over a CM extension $F$  in the sense we are using.
(On the other hand, both conjecturally and experimentally (see~\cite{Fig}),
$\rhobar$  \emph{is} modular for $\GL(2)/F$ if we omit the
self-dual requirement.)

\medskip

\Remark
For $n = 4$, suppose that $E$ be a totally real field, and $\rho: G_{E} \rightarrow  \GSp_4(\OL)$  
is a continuous irreducible Galois representation
unramified outside finitely many primes. Assume, otherwise, that $r$ satisfies all the
conditions of Theorem~\ref{theorem:evenodd}. 
Since $\rho$ is symplectic, $\wedge^2 \rho$ has a one
dimensional summand.   Let $r$ be the complementary summand. 
It is a simple exercise to see that $r$ is self-dual up to twist by a character
that is either totally odd or totally even, and that $r$ is ordinary
with distinct Hodge--Tate weights. 
If $\rhobar$
has image containing $\GSp_4(\F)$, then the image of $\rbar$ in $\GL_5(\F)$ is
presumably large (this is an unpleasant calculation that the author has no interest
in attempting). If so, then one may apply Theorem~\ref{theorem:evenodd} to deduce
that $r$ is modular, and deduce from Theorem~\ref{theorem:totallyodd2} that $r$ is odd.
Consequently  $\rhobar$ cannot be totally even, that is,
$\rhobar(c)$ is not a scalar for any complex conjugation $c \in G_E$.
This result seems to be about the natural limit for such arguments --- there does not
seem to be a way to deduce anything about totally even representations with
image in $\GSp_6(\OL)$, for example.

\section{Compatible Families}

Let us recall from~\cite{TaylorFM} the notion of a weakly compatible family of Galois representations.
(In this section only, $\OL$ will be denote a global ring of integers, not a local one.)

\begin{df} A weakly compatible family  $R = (L,\{\rho_{\lambda}\},P_{\ell}(T),S,\{m,n\})$ of two dimensional Galois representations over $\Q$  consists of:
\begin{enumerate}
\item A number field $L$ with ring of integers $\OL$,
\item A finite set of rational primes $S$,
\item For each prime $\ell \not\in S$, a monic polynomial $P_{\ell}(T)$ of degree $2$ with coefficients in $L$.
\item For each prime $\lambda$ of $\OL$ with residue characteristic $p$,
$$\rho_{\lambda}: \Gal(\Qbar/\Q) \rightarrow \GL_2(\OL_{\lambda})$$
is a continuous representations such that, if $p \not\in S$, then
$\rho_{\lambda} | D_{p}$ is crystalline, and if $\ell  \not\in S \cup \{ p\}$ then
$\rho_{\lambda}$ is unramified at $\ell$ and
$\rho_{\lambda}(\Frob_{\ell})$ has characteristic polynomial $P_{\ell}(T)$, 
\item $m$ and $n$ are integers such that for all primes
$\lambda$ of $\OL_{F}$ above $p$, the representation $\rho_{\lambda}|D_p$
is Hodge--Tate with Hodge--Tate weights $m$ and $n$.
\end{enumerate}
\end{df}

Say that $R$ is irreducible if one (respectively, any) $\rho_{\lambda}$ is irreducible.
We prove the following result (contrast the result of Kisin~\cite{Kisin}, Corollary~(0.5)).

\begin{theorem} Let $R$ be an irreducible weakly compatible family of
two dimensional Galois representations of $\Q$ with
distinct Hodge--Tate weights.
Then, up to twist,
$R$ arises from a rank $2$ Grothendieck motive $M(f)$ attached to a classical
modular form $f$ of weight $\ge 2$.
\end{theorem}

\begin{proof} 

If $\{\rho_{\lambda}\}$ is odd (for any $\lambda$), this is a consequence of~\cite{Kisin}, Corollary~(0.5).
Hence, we may assume that
$k = n - m > 0$ and that $\rho_{\lambda}$ is even for all $\lambda$. Without loss of generality, we may
assume that $m = 0$. If the projective image of $\rhobar_{\lambda}$ is either $\{A_4,S_4,A_5\}$ for
infinitely many $\lambda$, then the projective image of $\rho_{\lambda}$ is also finite and we are in case $2$ above.
If the projective image of $\rhobar_{\lambda}$ is dihedral for infinitely many $\lambda$, then
$\rho_{\lambda}$  is induced from a one dimensional character of a quadratic extension of $\Q$. In the latter case, the modularity of $\rho_{\lambda}$ is
 an easy consequence of class field theory (see~\cite{FM}).
 Hence, we may assume that there exist infinitely many primes $\lambda \in \OL$ such that
 $\OL_{\lambda} = \Q_{p}$, and such that
 the image of $\rhobar_{\lambda}$ contains $\SL_2(\F_p)$. Consider a sufficiently large such prime 
 $\lambda$.
 If  $\rho_{\lambda}|D_{p}$ is ordinary, 
 then Theorem~\ref{theorem:even} implies that $\rho_{\lambda}$ is odd, a contradiction. 
 Otherwise, since we may assume that $p \notin S$, the representation $\rho_{\lambda}|D_{p}$
 is crystalline at $p$, and (assuming that $p$ is sufficiently large with respect to $k$) that
 $\rhobar_{\lambda} |D_{p} \simeq \Ind(\omega^k_2)$ for the fundamental tame character $\omega_2$
of  level  $2$. Consider the representation $r = \ad^0(\rhobar_{\lambda}) \otimes \eps$, where
$\eps$ is the cyclotomic character. The
proof of Theorem~\ref{theorem:evenodd} can be modified to show that for $p$ sufficiently large, $r$ is modular and hence $\rhobar_{\lambda}$ is odd, completing the proof of the theorem.  The key adjustment required to prove the potential modularity of
$r$  is to replace the appeal  to Theorem~7.5 of~\cite{BLGHT} with
Theorem~7.6 of \emph{ibid}, noting that the coefficients of the representation $\rhobar_{\lambda}$
are $\Q_{p}$, and that we may take $\lambda$ sufficiently large so as to deduce from
Lemma~7.4 of~\cite{BLGHT} that $\rbar$ has $2k$-big image.  
\end{proof}

\section{Complements}

One idea of this paper is to use potential modularity and functoriality to rule
out the existence of Galois representations whose infinity type has the same infinitesimal
character as a non-unitary finite dimensional representation. This method, however,
cannot be  applied in all such situations.
Consider a 
representation $\rho: G_{\Q} \rightarrow  \GL_3(E)$
 with three distinct Hodge--Tate weights that are \emph{not} in arithmetic
progression.
The $8$ dimensional irreducible
subrepresentation $\psi$ of $\rho \otimes \rho^\vee$, which one might hope
to prove is
automorphic for $\GL(8)/F^{+}$, does not have distinct Hodge--Tate weights.
Moreover, even if one knew that $\psi$ arose from some automorphic form
$\Pi_{F^{+}}$, functoriality is not sufficiently developed to reconstruct the 
 $\pi_{F^{+}}$ (associated
to $\rho$) from $\Pi_{F^{+}}$. 
Another natural question that falls outside the scope of our methods is the following.
\begin{question} Let $K$ be an imaginary quadratic field, and let $E/K$ be an elliptic curve. Does
there exist  a totally real field $F^{+}$ such that $E$ is potentially modular over $F:=F^{+}.K$? That is,
 does there exist an automorphic representation $\pi$ for $\GL(2)/F$ such that $L(\pi_F,s) = L(E/F,s)$?
\end{question}

\section{$\psibar$ has big image}
\label{section:appendix}

In this final, technical section, we verify that the residual 
representation $\psibar$ occuring in Lemma~\ref{lemma:bigrefer}
has $2$-big image.
 The definition of $2$-big image depends
on the residue field $\F$. However, in all modularity lifting theorems it is
harmless to extend scalars. (In particular, in Theorem~7.5 of~\cite{BLGHT}, one
is free to replace $\OL = \OL_L$ by $\OL_M$ for any finite extension $M/L$.)

\begin{lemma} If $\F_{p^2} \subset \F$, the residual representation $\psibar$ of section~\ref{section:ramakrishnan}
has $2$-big image
in the sense of~\cite{BLGHT}, Definition~7.2. \label{lemma:bigimage}
\end{lemma}

\begin{proof} 
Write $V$ for the $4$-dimensional vector space underlying $\psibar$.
Recall that $G = \mathrm{im}(\psi) \subseteq \GL_4(\F)$ has $2$-big image if:
\begin{enumerate}
\item $G$ has no $p$-power quotient.
\item $H^i(G,\ad^0(V)) = 0$ for $i = 0$ and $i = 1$.
\item For every irreducible $G$-submodule $W$ of $\ad^0(V)
\subset \Hom(V,V)$, there
exists an $g \in G$ such that:
\begin{enumerate} 
\item The $\alpha$-generalized eigenspace $V_{g,\alpha}$
for $g$ is one dimensional for some $\alpha \in \F$.
\item For every other eigenvalue $\beta$ of $g$, $\beta^2 \ne \alpha^2$.
\item If $v \in V_{g,\alpha}$,  there is an inclusion
$v \in W(v)$.
\end{enumerate}
\end{enumerate}
If this property holds for a finite index subgroup $H \subseteq G$ of index
co-prime to $p$, then it holds for $G$.

By assumption, if $G$ is the image of $\psibar$, then
$G$ contains, with index co-prime to $p$, the group $H = \SL_2(k)^2 \rtimes \Z/2\Z$
for some field $\F_p \subset k \subset \F$. Condition (1) holds for $H$.

Write $V$ for the $4$-dimensional vector space underlying $\psibar$.
There  is a natural decomposition $V = X \otimes Y$ where the first $\SL_2(k)$
factor acts
on $X$ and the second on $Y$.
We find that 
$$\ad^0(V) =  W_6 \oplus W_{9}$$
for irreducible representations $W_6$ and $W_9$ of dimensions $6$ and $9$, respectively. Explicitly,
\begin{enumerate}
\item 
$W_6 = \ad^0(X) \oplus \ad^0(Y)$, 
\item $W_9 = \ad^0(X) \otimes \ad^0(Y)$.
\end{enumerate}
Let $\SC$ denote a split Cartan subgroup of $\SL_2(\F_p)$, and $\NS$ a non-split Cartan
subgroup of $\SL_2(\F_p)$.
There are isomorphisms
$$ \SC \simeq \F^{\times}_p, \qquad \NS \simeq (\F^{\times}_{p^2})^{\sigma = -1},$$
where $\sigma$ denotes Frobenius in $\F^{\times}_{p^2}$. 
Let $\delta$ denote a generator of $\SC$ and $\gamma$ a generator of
$\NS$. Then $g = (\delta,\gamma) \subset \SC \times \NS \subset \SL_2(\F_p)^2$ is naturally an element of $H$.
The eigenvalues of $g$ on $X \otimes Y$ 
are $\{ \delta \gamma, \delta \gamma^{-1}, \delta^{-1} \gamma, \delta^{-1} \gamma^{-1}\}$.
Suppose that the squares of two distinct elements of this set coincide.
Then there must be an equality
$$\delta^{4i} \gamma^{4j} = 1$$
where $i$ and $j$ are in $\{-1,0,1\}$ and are not both zero. 
Taking the product (respectively ratio) of this
element and its conjugate
$\delta^{4i} \gamma^{-4j}$, we deduce that either $\delta^8 = 1$ or $\gamma^8 =1$.
Yet $\delta$ and $\gamma$ have orders $p-1$ and $p+1$ respectively,
contradicting the assumption that $p > 7$. In particular, the squares of the eigenvalues
(and thus the eigenvalues) of $g$ are distinct.
Suppose that $W \subset \ad^0(V)$ is an irreducible representation such that
$W$ contains a non-zero $g$-invariant element $w$. Then $w$ induces a map
$w:V \rightarrow V$ which is $g$-equivariant.  Since all the eigenvalues of $g$ are
distinct, all the eigenvectors of $g$ are eigenvectors of $w$.
In particular, since $w$ is non-zero, there exists at least one eigenvector of $g$ with
eigenvalue $\alpha$ which does not lie in the kernel of $w$. If $V_{g,\alpha}$ is
the line generated by this eigenvector, and $v \in V_{g,\alpha}$, it follows that
$v$ is a non-zero multiple of $w(v)$ and thus $v \in W(v)$. By inspection, both
$W_6$ and $W_9$ contain $g$-invariant elements, and thus we have verified condition (3).
	
\medskip

We now verify condition (2). Since $\SL_2(k)^2$ is contained in $G$ with
index co-prime to $p$, it suffices to show that the cohomology vanishes for this group. This
is elementary for $H^0$.
Write $H = \Gamma_1 \times \Gamma_2$ where $\Gamma_1 = \SL_2(k)$ acts on $X$ and
$\Gamma_2 = \SL_2(k)$ acts on $Y$.
The modules $W_6$ and $W_9$ decompose as $H$ modules of the form
$A \otimes B$ where $A = \Sym^{2i}(X)$ and $B = \Sym^{2j}(Y)$ for 
$i$ and $j$ in $\{0,1\}$, where at least one of $A$ and $B$ is non-trivial.
By symmetry, we may assume that $B$ is non-trivial. Hence
$A \otimes B$ has no invariants as an $\Gamma_2$-module, and thus, by 
inflation-restriction,
$$H^1(\Gamma_1 \times \Gamma_2,A \otimes B) \hookrightarrow H^1(\Gamma_1,A \otimes B) =
 B \otimes H^1(\Gamma_1,A) = 
B \otimes H^1(\SL_2(k),\Sym^{2i} X) = 0,$$
as the latter group vanishes for $i = 0,1$ and $p > 5$.
\end{proof}

\end{document}